\documentclass[11pt, notitlepage]{article}   % The document class has to be "report"

\usepackage{amsmath,amsthm,amsfonts}   % These are to write mathematics

\usepackage{amscd}
\usepackage{amsfonts}
\usepackage{amssymb}
\usepackage{color}      % Need the color package
\usepackage{epsfig}
\usepackage{graphicx}           % Graphics package

\usepackage{tikz}
\usepackage{tkz-graph}
\usepackage{tkz-berge}

\theoremstyle{plain}
\newtheorem{theorem}{Theorem}[section]
\newtheorem{lemma}[theorem]{Lemma}
\newtheorem{corollary}[theorem]{Corollary}
\newtheorem{proposition}[theorem]{Proposition}
\newtheorem{observation}[theorem]{Observation}

\newtheorem{question}[theorem]{Question}
\newtheorem{remark}[theorem]{Remark}

\theoremstyle{definition}

\newcommand{\diam}{\textnormal{diam}}

\def\finf{\mathop{{\rm I}\kern -.27 em {\rm F}}\nolimits}

\newcommand{\bdim}{\textnormal{bdim}}
\newcommand{\adim}{\textnormal{adim}}
\newcommand{\supp}{\textnormal{supp}}

\begin{document}

%\date{}

\title{Broadcast Dimension of Graphs}

\author{{\bf{Jesse Geneson}}$^1$ and {\bf{Eunjeong Yi}}$^2$\\
\small Iowa State University, Ames, IA 50011, USA$^1$\\
\small Texas A\&M University at Galveston, Galveston, TX 77553, USA$^{2}$\\
{\small\em geneson@iastate.edu}$^1$; {\small\em yie@tamug.edu}$^2$}

\maketitle

\date{}

\begin{abstract}
In this paper we initiate the study of broadcast dimension, a variant of metric dimension. Let $G$ be a graph with vertex set $V(G)$, and let $d(u,w)$ denote the length of a $u-w$ geodesic in $G$. For $k \ge 1$, let $d_k(x,y)=\min \{d(x,y), k+1\}$. A function $f: V(G) \rightarrow \mathbb{Z}^+ \cup \{0\}$ is called a \emph{resolving broadcast} of $G$ if, for any distinct $x,y \in V(G)$, there exists a vertex $z \in V(G)$ such that $f(z)=i>0$ and $d_{i}(x,z) \neq d_{i}(y,z)$. The \emph{broadcast dimension}, $\bdim(G)$, of $G$ is the minimum of $c_f(G)=\sum_{v \in V(G)} f(v)$ over all resolving broadcasts of $G$, where $c_f(G)$ can be viewed as the total cost of the transmitters (of various strength) used in resolving the entire network described by the graph $G$. Note that $\bdim(G)$ reduces to $\adim(G)$ (the adjacency dimension of $G$, introduced by Jannesari and Omoomi in 2012) if the codomain of resolving broadcasts is restricted to $\{0,1\}$. We determine its value for cycles, paths, and other families of graphs. We prove that $\bdim(G) = \Omega(\log{n})$ for all graphs $G$ of order $n$, and that the result is sharp up to a constant factor. We show that $\frac{\adim(G)}{\bdim(G)}$ and $\frac{\bdim(G)}{\dim(G)}$ can both be arbitrarily large, where $\dim(G)$ denotes the metric dimension of $G$. We also examine the effect of vertex deletion on the adjacency dimension and the broadcast dimension of graphs.
\end{abstract}


\section{Introduction}

Let $G$ be a finite, simple, and undirected graph with vertex set $V(G)$ and edge set $E(G)$. The \emph{distance} between two vertices $x, y \in V(G)$, denoted by $d(x, y)$, is the length of a shortest path between $x$ and $y$ in $G$; if $x$ and $y$ belong to different components of $G$, we define $d(x,y)=\infty$. Metric dimension, introduced by Slater~\cite{slater} and by Harary and Melter~\cite{harary}, is a graph parameter that has been studied extensively. A vertex $z \in V(G)$ \emph{resolves} a pair of vertices $x,y \in V(G)$ if $d(x,z) \neq d(y,z)$. A set $S \subseteq V(G)$ is a \emph{resolving set} of $G$ if, for any distinct $x, y \in V(G)$, there exists $z \in S$ such that $d(x,z) \neq d(y,z)$. The \emph{metric dimension} of $G$, denoted by $\dim(G)$, is the minimum cardinality over all resolving sets of $G$. Khuller et al.~\cite{tree2} considered robot navigation as one of the applications of metric dimension, where a robot that moves from node to node knows its distances to all landmarks.

For $x\in V(G)$ and $S \subseteq V(G)$, let $d(x, S)=\min\{d(x,y) : y \in S\}$. Meir and Moon~\cite{distance_dom} introduced distance-$k$ domination. For a positive integer $k$, a set $D \subseteq V(G)$ is called a \emph{distance-$k$ dominating set} of $G$ if, for each $u \in V(G)- D$, $d(u, D) \le k$. The \emph{distance-$k$ domination number} of $G$, denoted by $\gamma_k(G)$, is the minimum cardinality over all distance-$k$ dominating sets of $G$; the distance-$1$ domination number is the well-known domination number. Erwin~\cite{erwin1, erwin2} introduced the concept of broadcast domination, where cities with broadcast stations have transmission power that enable them to broadcast messages to cities at distances greater than one, depending on the transmission power of broadcast stations. More explicitly, following~\cite{erwin1, erwin2}, a function $f: V(G) \rightarrow \{0,1,2,\ldots, \diam(G)\}$ is called a \emph{dominating broadcast} of $G$ if, for each vertex $x \in V(G)$, there exists a vertex $y \in V(G)$ such that $f(y)>0$ and $d(x,y) \le f(y)$. The \emph{broadcast (domination) number}, $\gamma_b(G)$, of G is the minimum of $D_f(G):=\sum_{v\in V(G)} f (v)$ over all dominating broadcasts $f$ of $G$; here, $D_f(G)$ can be viewed as the total cost of the transmitters used to achieve full coverage of a network of cities described via the graph $G$ being considered. Note that $\gamma_b(G)$ reduces to $k \cdot \gamma_k(G)$ if the codomain of dominating broadcasts is restricted to $\{0,k\}$. It is known that determining the domination number of a general graph is an NP-hard problem (see~\cite{NP}).

 Recently, Jannesari and Omoomi~\cite{adjacency_dim} introduced \emph{adjacency dimension} of $G$, denoted by $\adim(G)$, as a tool to study the metric dimension of lexicographic product graphs; they defined the adjacency distance between two vertices $x,y\in V(G)$ to be $0,1,2$, respectively, if $d(x,y)=0$, $d(x,y)=1$, and $d(x,y) \ge 2$. Adjacency resolving set and adjacency dimension are defined analogously in~\cite{adjacency_dim}. Assuming that a landmark that can detect long distance can be costly, the authors of~\cite{adjacency_dim} considered a robot that detects its position only from landmarks adjacent to it; this can be viewed as combining the concept of a resolving set and a dominating set. More generally, we can apply the concept of a distance-$k$ dominating set to a resolving set. If a robot can detect up to distance $k>0$ from each landmark, the minimum number of such landmarks to determine the robot's position on the graph is called the distance-$k$ dimension of $G$, denoted by $\dim_k(G)$; note that $\dim_1(G)=\adim(G)$. 

Now, we apply the concept of a dominating broadcast to a resolving set. For a positive integer $k$ and for $x,y \in V(G)$, let $d_k(x,y)=\min\{d(x,y),k+1\}$. Let $f: V(G) \rightarrow \mathbb{Z}^+ \cup \{0\}$ be a function. We define $\supp_G(f) = \left\{v \in V(G): f(v) > 0 \right\}$. We say that $f$ is a \emph{resolving broadcast} of $G$ if, for any distinct $x,y \in V(G)$, there exists a vertex $z \in \supp_G(f)$ such that $d_{f(z)}(x,z) \neq d_{f(z)}(y,z)$. The \emph{broadcast dimension} of $G$, denoted by $\bdim(G)$, is the minimum of $c_f(G)=\sum_{v \in V(G)} f(v)$ over all resolving broadcasts $f$ of $G$, where $c_f(G)$ can be viewed as the total cost of the transmitters (of various strength) used in resolving the entire network described via the graph $G$ being considered. Note that, if the codomain of resolving broadcasts is restricted to $\{0,k\}$, where $k$ is a positive integer, then $\bdim(G)$ reduces to $k \cdot \dim_k(G)$. For an ordered set $S=\{u_1, u_2, \ldots, u_k\} \subseteq V(G)$ of distinct vertices, the metric code, the adjacency code, and the broadcast code, respectively, of $v \in V(G)$ with respect to $S$ are the $k$-vectors $r_S(v) =(d(v, u_1), d(v, u_2), \ldots, d(v, u_k))$, $a_S(v) =(d_1(v, u_1), d_1(v, u_2), \ldots, d_1(v, u_k))$, and $b_S(v) =(d_{i_1}(v, u_1), d_{i_2}(v, u_2), \ldots, d_{i_k}(v, u_k))$, where $f(u_j)=i_j>0$ for a resolving broadcast $f$ being considered. It is known that determining the metric dimension (adjacency dimension, respectively) of a graph is an NP-hard problem; see~\cite{NP} (\cite{Juan}, respectively).

Suppose $f(x)$ and $g(x)$ are two functions defined on some subset of real numbers. We write $f(x)=O(g(x))$ if there exist positive constants $N$ and $C$ such that $|f(x)| \le C |g(x)|$ for all $x >N$, $f(x)=\Omega(g(x))$ if $g(x)=O(f(x))$, and $f(x)=\Theta(g(x))$ if $f(x)=O(g(x))$ and $f(x)=\Omega(g(x))$. 

In this paper, we initiate the study of broadcast dimension. In Section \ref{s:genres}, we discuss some general results on the metric dimension, the adjacency dimension, and the broadcast dimension of graphs. For example, it is easy to see that for any graph $G$, $\dim(G) \le \bdim(G) \le \adim(G)$. We also find the broadcast dimension of paths and cycles. In Section \ref{s:extremal}, we prove that $\bdim(G) = \Omega(\log{n})$ for all graphs $G$ of order $n$, and that the result is sharp up to a constant factor. We also characterize the family of graphs of adjacency dimension $k$ for each $k$. In Section \ref{s:highlow}, we characterize the graphs $G$ such that $\bdim(G)$ equals 1, 2, and $|V(G)|-1$. It is noteworthy that $\bdim(G)=2$ ($\adim(G)=2$, respectively) implies that $G$ is planar, whereas an example of non-planar graph $G$ with $\dim(G)=2$ was given in~\cite{tree2}. In Section \ref{s:comparing}, we provide graphs $G$ such that both $\adim(G)-\bdim(G)$ and $\bdim(G)-\dim(G)$ can be arbitrarily large. We also show that, for two connected graphs $G$ and $H$ with $H \subset G$, $\dim(H)-\dim(G)$ ($\bdim(H)-\bdim(G)$ and $\adim(H)-\adim(G)$, respectively) can be arbitrarily large. In addition, we find all trees $T$ such that $\bdim(T) = \dim(T)$. In Section \ref{s:delete}, we examine the effect of vertex deletion on adjacency dimension and broadcast dimension. We also investigate the effect of edge deletion on adjacency dimension. In Section \ref{s:open}, we conclude with some open problems.

We conclude the introduction with some terminology and notation that we will use throughout the paper. The \emph{diameter}, $\diam(G)$, of $G$ is $\max\{d(x,y): x,y \in V(G)\}$. The \emph{open neighborhood} of a vertex $v \in V(G)$ is $N(v)=\{u \in V(G) : uv \in E(G)\}$ and its \emph{closed neighborhood} is $N[v]=N(v) \cup \{v\}$. The \emph{degree} of a vertex $u$ in $G$, denoted by $\deg(u)$, is $|N(u)|$. An \emph{end vertex} is a vertex of degree one, and a \emph{major vertex} is a vertex of degree at least three. The \emph{join} of two graphs $H_1$ and $H_2$, denoted by $H_1 +H_2$, is the graph obtained from the disjoint union of two graphs $H_1$ and $H_2$ by joining every vertex of $H_1$ with every vertex of $H_2$. We denote by $P_n$, $C_n$, $K_n$, and $K_{m,n}$ respectively the path, cycle, and complete graph on $n$ vertices, and the complete bipartite graph with parts of size $m$ and $n$. We denote by ${\bf 1}_{\alpha}$ and ${\bf 2}_{\alpha}$, respectively, the $\alpha$-vector with 1 on each entry and the $\alpha$-vector with 2 on each entry.

\section{General results}\label{s:genres}

In this section, we discuss some general results on the metric dimension, the adjacency dimension, and the broadcast dimension of graphs. We also determine the broadcast dimension of paths and cycles. 
For distinct $u,w\in V(G)$, if $N(u)-\{w\}=N(w)-\{u\}$, then  $u$ and $w$ are called \emph{twin vertices} of $G$.

\begin{observation}\label{obs_twin}
Let $u$ and $w$ be twin vertices of a graph $G$. Then 
\begin{itemize}
\item[(a)] \emph{\cite{Hernando}} for any resolving set $S$ of $G$, $S \cap \{u, w\} \neq \emptyset$;
\item[(b)] \emph{\cite{adjacency_dim}} for any adjacency resolving set $A$ of $G$, $A \cap \{u, w\} \neq \emptyset$;
\item[(c)] for any resolving broadcast $f$ of $G$, $f(u)>0$ or $f(w)>0$.
\end{itemize}
\end{observation}

\begin{proposition}\emph{\cite{adjacency_dim}}\label{adj_bounds}
\begin{itemize}
\item[(a)] If $G$ is a connected graph, then $\adim(G) \ge \dim(G)$.
\item[(b)] If $G$ is a connected graph with $\diam(G)=2$, then $\adim(G)=\dim(G)$. Moreover, there exists a graph $G$ such that $\adim(G)=\dim(G)$ and $\diam(G)>2$.
\item[(c)] For every graph $G$, $\adim(G)=\adim(\overline{G})$, where $\overline{G}$ denotes the complement of $G$.
\end{itemize}
\end{proposition}

\begin{observation}\label{obs_bdim}
\begin{itemize}
\item[(a)] For any graph $G$ of order $n \ge 2$, $1 \le \dim(G) \le \bdim(G) \le \adim(G) \le n-1$.
\item[(b)] For any graph $G$ with $\diam(G) \in \{1,2\}$, $\dim(G)=\bdim(G)=\adim(G)$.
\end{itemize}
\end{observation}

Next, we consider graphs $G$ with $\diam(G) \le 2$. For two graphs $H_1$ and $H_2$, $\diam(H_1+H_2) \le 2$; thus, $\dim(H_1+H_2)=\bdim(H_1+H_2)=\adim(H_1+H_2)$ by Observation~\ref{obs_bdim}(b).

\begin{theorem}\emph{\cite{wheel1, wheel2}}\label{dim_wheel}
For $n \ge 3$, 
\begin{equation*}
\dim(C_n+K_1)=\left\{
\begin{array}{ll}
3 & \mbox{ if } n \in \{3,6\},\\
\lfloor \frac{2n+2}{5}\rfloor & \mbox{ otherwise.} 
\end{array}\right.
\end{equation*}
\end{theorem}

\begin{theorem}\emph{\cite{fan}}\label{dim_fan}
For $n \ge 1$, 
\begin{equation*}
\dim(P_n+K_1)=\left\{
\begin{array}{ll}
1 & \mbox{ if } n=1,\\
2 & \mbox{ if } n \in \{2,3\},\\
3 & \mbox{ if } n=6,\\
\lfloor \frac{2n+2}{5}\rfloor & \mbox{ otherwise.} 
\end{array}\right.
\end{equation*}
\end{theorem}

Proposition~\ref{adj_bounds}(b), along with Theorems~\ref{dim_wheel} and~\ref{dim_fan}, implies the following proposition.

\begin{proposition}\emph{\cite{adjacency_dim}}\label{adim_wheel}
For $n \ge 7$, if $G \in \{P_n, C_n\}$, then $\adim(G+K_1)=\lfloor\frac{2n+2}{5}\rfloor$.
\end{proposition}

As an immediate consequence of Observation~\ref{obs_bdim}(b) and Theorems~\ref{dim_wheel} and~\ref{dim_fan}, we have the following corollary.

\begin{corollary}
For $n \ge 3$, let $G \in \{P_n, C_n\}$. Then \begin{equation*}
\bdim(G+K_1)=\left\{
\begin{array}{ll}
2 & \mbox{ if } n=3 \mbox{ and } G=P_3,\\
3 & \mbox{ if } n=3 \mbox{ and } G=C_3,\\
3 & \mbox{ if } n=6,\\
\lfloor \frac{2n+2}{5}\rfloor & \mbox{ otherwise.} 
\end{array}\right.
\end{equation*}
\end{corollary}

The metric dimension and the adjacency dimension, respectively, of a complete k-partite graphs was determined in~\cite{kpartite} and~\cite{adjacency_dim}.

\begin{theorem}\emph{\cite{adjacency_dim, kpartite}}\label{adim_kpartite}
For $k \ge 2$, let $G=K_{a_1, a_2, \ldots, a_k}$ be a complete $k$-partite graph of order $n=\sum_{i=1}^{k}a_i$. Let $s$ be the number of partite sets of $G$ consisting of exactly one element. Then 
\begin{equation*}
\dim(G)=\adim(G)=\left\{
\begin{array}{ll}
n-k & \mbox{ if } s=0,\\
n+s-k-1 & \mbox{ if } s \neq 0.
\end{array}\right.
\end{equation*}
\end{theorem}

As an immediate consequence of Observation~\ref{obs_bdim}(b) and Theorem~\ref{adim_kpartite}, we have the following corollary.

\begin{corollary}
For $k \ge 2$, let $G=K_{a_1, a_2, \ldots, a_k}$ be a complete $k$-partite graph of order $n=\sum_{i=1}^{k}a_i$. Let $s$ be the number of partite sets of $G$ consisting of exactly one element. Then 
\begin{equation*}
\bdim(G)=\left\{
\begin{array}{ll}
n-k & \mbox{ if } s=0,\\
n+s-k-1 & \mbox{ if } s \neq 0.
\end{array}\right.
\end{equation*}
\end{corollary}

Now, we recall the metric dimension of the Petersen graph.

\begin{theorem}\emph{\cite{petersen}}\label{dim_petersen}
For the Petersen graph $\mathcal{P}$, $\dim(\mathcal{P})=3$.
\end{theorem}

Since $\diam(\mathcal{P})=2$, Observation~\ref{obs_bdim}(b) and Theorem~\ref{dim_petersen} imply the following corollary.

\begin{corollary}
For the Petersen graph $\mathcal{P}$, $\bdim(\mathcal{P})=\adim(\mathcal{P})=3$.
\end{corollary}

Next, we consider paths and cycles.

\begin{proposition}\emph{\cite{adjacency_dim}}\label{adim_pc}
For $n \ge 4$, $\adim(P_n)=\adim(C_n)=\lfloor \frac{2n+2}{5}\rfloor$.
\end{proposition}

\begin{theorem}
For $n \geq 4$, $\bdim(P_n) = \bdim(C_n) = \lfloor \frac{2n+2}{5}\rfloor$
\end{theorem}

\begin{proof}
Let $G$ be $P_n$ or $C_n$, with vertices $v_0, \dots, v_{n-1}$ in order, where $n\ge 4$. By Observation~\ref{obs_bdim}(a) and Proposition~\ref{adim_pc}, $\bdim(G) \le \lfloor \frac{2n+2}{5}\rfloor$ for $n \geq 4$. Thus it suffices to prove that $\sum_{v \in V(G)} f(v)$ is minimized when $f(v) \leq 1$ for all $v \in V(G)$.

Suppose that $f$ is a resolving broadcast that achieves $\bdim(G)$. If $f(v) \leq 1$ for all $v \in V(G)$, then we are done. Otherwise, we modify $f$ to obtain a new resolving broadcast $f'$ for which $\sum_{v \in V(G)} f'(v) \leq \sum_{v \in V(G)} f(v)$ and $f'(v) \leq 1$ for all $v \in V(G)$. 

Start by defining $f_0$ such that $f_0(v) = f(v)$ for all $v \in V(G)$. Given $f_i$, let $v_j$ be any vertex in $V(G)$ such that $f_i(v_j) > 1$. If $f_i(v_j) = 2$, then we define $f_{i+1}(v_{(j-1)\mod n}) = f_{i+1}(v_{(j+1)\mod n}) = 1$ and $f_{i+1}(v_j) = 0$, unless $v_j$ is an end vertex of $P_n$. If $G = P_n$ and $f_i(v_0) = 2$, then we define $f_{i+1}(v_0) = 1$ and $f_{i+1}(v_{1}) = 1$. If $G = P_n$ and $f_i(v_{n-1}) = 2$, then we define $f_{i+1}(v_{n-1}) = 1$ and $f_{i+1}(v_{n-2}) = 1$. 

Otherwise if $f_i(v_j) = x > 2$, then we define $f_{i+1}(v_j) = x-2$ and $f_{i+1}(v_{(j-x+1)\mod n}) = f_{i+1}(v_{(j+x-1)\mod n}) = 1$. If any vertices are assigned multiple values for $f_{i+1}$, only the maximum value is used. If any vertex $v$ is assigned no values for $f_{i+1}$, then $f_{i+1}(v) = f_i(v)$.

The process will end in finitely many steps, so suppose that $k$ is an integer such that $f_k(v) \leq 1$ for all $v \in V(G)$. Then we let $f' = f_k$, and $\sum_{v \in V(G)} f'(v) \leq \sum_{v \in V(G)} f(v)$ by construction. Thus $\bdim(P_n) = \adim(P_n) = \bdim(C_n) = \adim(C_n) = \lfloor \frac{2n+2}{5}\rfloor$.~\hfill
\end{proof}

\section{Extremal bounds and characterization}\label{s:extremal}
In this section, we prove that $\bdim(G) = \Omega(\log{n})$ for all graphs $G$ of order $n$, and that the result is sharp up to a constant factor. We also obtain bounds for the clique number and maximum degree of graphs with adjacency dimension $k$ or broadcast dimension $k$. Furthermore, we characterize the family of graphs of adjacency dimension $k$. First, we recall some known bounds for the metric dimension of graphs.

\begin{theorem}\emph{\cite{tree1}}\label{dim_bounds}
For a connected graph $G$ of order $n \ge 2$ and diameter $d$, 
$$f(n,d) \le \dim(G) \le n-d,$$
where $f(n,d)$ is the least positive integer $k$ for which $k+d^k \ge n$.
\end{theorem} 

Hernando et al. \cite{Hernando} improved the bound in Theorem~\ref{dim_bounds}.

\begin{theorem}\emph{\cite{Hernando}}\label{dim_bound2}
Let $G$ be a connected graph of order $n$, diameter $d $, and $\dim(G)=k$. Then
$$n \le \left(\left\lfloor\frac{2d}{3}\right\rfloor+1\right)^k+k\sum_{i=1}^{\lceil\frac{d}{3}\rceil}(2i-1)^{k-1}.$$ 
\end{theorem}

As a corollary of Observation~\ref{obs_bdim}(a) and Theorem~\ref{dim_bound2}, we obtain bounds on the maximum order of any graph $G$ with $\diam(G)=d$ and $\bdim(G) = k$.

\begin{corollary}\label{cor_bound1}
For any graph $G$ with $\diam(G) = d$ and $\bdim(G) = k$, 
$$|V(G)| \leq \left(\left\lfloor \frac{2d}{3}\right\rfloor +1\right)^{k}+k \sum_{i = 1}^{\lceil \frac{d}{3}\rceil} (2i-1)^{k-1}.$$
\end{corollary}

\begin{proof}
If $G$ has $\bdim(G) = k$, then $\dim(G) \leq k$ by Observation~\ref{obs_bdim}(a). So, the desired result follows from Theorem~\ref{dim_bound2}.~\hfill
\end{proof}

We also obtain bounds on the maximum order of any subgraph of $G$ with $\diam(G) = d$ and $\bdim(G) = k$.

\begin{theorem}\emph{\cite{mdapa}}\label{dim_subgraph}
For any graph $G$ with $\dim(G)=k$ and any subgraph $H$ of $G$ with $\diam(H)=d$, $|V(H)|\le (d+1)^k$. 
\end{theorem}

\begin{corollary}\label{cor_bound2}
For any graph $G$ with $\bdim(G) = k$ and any subgraph $H$ of $G$ with $\diam(H) = d$, $|V(H)| \leq (d+1)^k$.
\end{corollary}

\begin{proof}
If $G$ has $\bdim(G) = k$, then $\dim(G) \leq k$ by Observation~\ref{obs_bdim}(a). So, the desired result follows from Theorem~\ref{dim_subgraph}.~\hfill
\end{proof}

\begin{remark}
By Observation~\ref{obs_bdim}(a), Corollaries~\ref{cor_bound1} and~\ref{cor_bound2} hold when $\bdim(G)=k$ is replaced by $\adim(G)=k$.
\end{remark}

The next result shows that Corollary~\ref{cor_bound1} is sharp for $d = 2$. This result uses a family of graphs from \cite{zubrilina, mdapa}.

\begin{theorem}\label{O_sharp}
There exist graphs $G$ of order $n$ with $\bdim(G) = O(\log n)$.
\end{theorem}

\begin{proof}
We construct a graph $G$ of order $n = k+2^k$ by starting with $k$ vertices $v_1, \dots, v_k$ in a clique, and adding $2^k$ new vertices $\left\{u_b\right\}_{b \in \left\{0,1\right\}^k}$ also in a clique labeled with binary strings of length $k$ such that $u_b$ has an edge with $v_j$ if and only if the $j^{th}$ digit of $b$ is $1$. 

Define the resolving broadcast $f$ such that $f(v_i) = 1$ for all $1 \leq i \leq k$ and $f(u_b) = 0$ for all $b \in \left\{0,1\right\}^k$. Since $n = k+2^k$ and $\sum_{v \in V(G)} f(v) = k$, we have $\bdim(G) = O(\log n)$. For any $n$ not of the form $k+2^k$, we can define $n'$ to be the least number greater than $n$ that is of the form $k+2^k$, construct $G'$ with $n'$ vertices as described, and delete any number of vertices $u_b$ from $G'$ until the remaining graph $G$ has $n$ vertices.~\hfill
\end{proof}

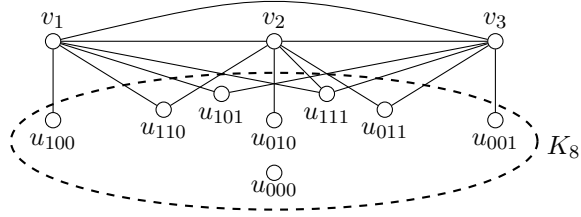
\begin{figure}[ht]
\centering
\begin{tikzpicture}[scale=.7, transform shape]

\node [draw, shape=circle, scale=.8] (a) at  (0, 0) {};
\node [draw, shape=circle, scale=.8] (b) at  (4.2, 0) {};
\node [draw, shape=circle, scale=.8] (c) at  (8.4, 0) {};
\node [draw, shape=circle, scale=.8] (1) at  (0, -1.5) {};
\node [draw, shape=circle, scale=.8] (2) at  (4.2, -1.5) {};
\node [draw, shape=circle, scale=.8] (3) at  (8.4, -1.5) {};
\node [draw, shape=circle, scale=.8] (4) at  (2.1, -1.3) {};
\node [draw, shape=circle, scale=.8] (5) at  (6.3, -1.3) {};

\node [draw, shape=circle, scale=.8] (6) at  (3.2, -1) {};
\node [draw, shape=circle, scale=.8] (7) at  (5.2, -1) {};
\node [draw, shape=circle, scale=.8] (8) at  (4.2, -2.5) {};

\node [scale=1.3] at (0,0.4) {$v_1$};
\node [scale=1.3] at (4.2,0.4) {$v_2$};
\node [scale=1.3] at (8.4,0.4) {$v_3$};

\node [scale=1.3] at (0, -1.9) {$u_{100}$};
\node [scale=1.3] at (4.2,-1.9) {$u_{010}$};
\node [scale=1.3] at (8.4,-1.9) {$u_{001}$};

\node [scale=1.3] at (2.1, -1.7) {$u_{110}$};
\node [scale=1.3] at (6.3, -1.7) {$u_{011}$};

\node [scale=1.3] at (3.2,-1.4) {$u_{101}$};
\node [scale=1.3] at (5.2,-1.4) {$u_{111}$};
\node [scale=1.3] at (4.2,-2.8) {$u_{000}$};
\node [scale=1.3] at (9.7,-2) {$K_8$};

\draw(a)--(b)--(c);\draw(a)--(1);\draw(b)--(2);\draw(c)--(3);\draw(a)--(4)--(b);\draw(b)--(5)--(c);\draw(a)--(6)--(c);\draw(a)--(7)--(c);\draw(b)--(7);\draw(a).. controls (4.2,1) .. (c);
\draw[thick,dashed] (4.2,-1.9) ellipse (5cm and 1.3cm);

\end{tikzpicture}
\caption{\small A graph $G$ of order $n$ satisfying $\bdim(G) = O(\log n)$; here $k=3$ for $G$ described in the proof of Theorem~\ref{O_sharp}.}\label{fig_bdim_O}
\end{figure}

Based on the proof of Theorem~\ref{O_sharp}, we have the following corollary.

\begin{corollary} \label{adim_olgn}
There exist graphs $G$ of order $n$ with $\adim(G)=O(\log n)$.
\end{corollary}

The construction in Theorem \ref{O_sharp} can also be used to recursively characterize the graphs $G$ with $\adim(G) = k$. Given any graph $G_1$ on $k$ vertices $v_1, \dots, v_k$ and $G_2$ on $2^k$ vertices $\left\{u_b\right\}_{b \in \left\{0,1\right\}^k}$, define the graph $B(G_1, G_2)$ to be obtained by connecting $v_i$ and $u_b$ if and only if the $i^{th}$ digit of $b$ is $1$. Moreover, define $\mathcal{B}(G_1, G_2)$ to be the family of induced subgraphs of $B(G_1, G_2)$ that contain every vertex in $G_1$. Finally, define $\mathcal{H}_0 = \emptyset$ and for each $k > 0$ define $\mathcal{H}_k$ to be the family of graphs obtained from taking the union of $\mathcal{B}(G_1, G_2)$ over all graphs $G_1$ with $j$ vertices $v_1, \dots, v_j$ and $G_2$ with $2^j$ vertices $\left\{u_b\right\}_{b \in \left\{0,1\right\}^j}$, for each $1 \leq j \leq k$.

\begin{theorem}\label{adimch}
For each $k \geq 1$, the set of graphs $G$ with $\adim(G) = k$ is $\mathcal{H}_k - \mathcal{H}_{k-1}$ up to isomorphism.
\end{theorem}

\begin{proof}
It suffices to show that the set of graphs $G$ with $\adim(G) \leq k$ is $\mathcal{H}_k$. By construction, every graph in $\mathcal{H}_k$ has $\adim(G) \leq k$, since the vertices $v_1, \dots, v_j$ are an adjacency resolving set. Thus it suffices to show that every graph $G$ with $\adim(G) \leq k$ is in $\mathcal{H}_k$. Fix an arbitrary graph $G$ with $\adim(G) \leq k$. Let $X = \left\{x_1, \dots, x_j\right\}$ be an adjacency resolving set for $G$ with $j \leq k$. Let $G_1$ be the induced subgraph of $G$ restricted to $X$, and let $G_2$ be the induced subgraph of $G$ restricted to $\overline{X}$. Label the vertex $v$ of $G_2$ as $u_b$ with a binary string $b$ so that the $i^{th}$ digit of $b$ is $1$ if and only if there is an edge between $v$ and $x_i$. Note that every vertex gets a unique label, or else $X$ would not be an adjacency resolving set. Let $G'_2$ be any graph on $2^j$ vertices $\left\{u_b\right\}_{b \in \left\{0,1\right\}^j}$ such that $G'_2|_{V(G_2)} = G_2$. Then $G$ is an induced subgraph of $B(G_1, G'_2)$ that contains every vertex in $G_1$, so $G$ is in $\mathcal{H}_k$.~\hfill
\end{proof}

As a corollary, we obtain an upper bound on the maximum order of a graph of adjacency dimension $k$. The graph in Theorem \ref{O_sharp} shows that the bound is sharp.

\begin{corollary}\label{max_order_adimk}
The maximum order of a graph of adjacency dimension $k$ is $k+2^k$. 
\end{corollary}

We also obtain a sharp upper bound on the maximum degree of a graph of adjacency dimension $k$.

\begin{corollary}
The maximum possible degree of any vertex in a graph of adjacency dimension $k$ is $k+2^k-1$.
\end{corollary}

\begin{proof}
The upper bound is immediate from Corollary \ref{max_order_adimk}, while the upper bound is achieved by the vertex $u_{1^k}$ in $B(K_k, K_{2^k})$.
\end{proof}

In addition, we obtain a sharp upper bound on the clique number of graphs of adjacency dimension $k$ and graphs of broadcast dimension $k$. 

\begin{corollary}
The maximum possible clique number of any graph of adjacency dimension $k$ is $2^k$. Similarly, the maximum possible clique number of any graph of broadcast dimension $k$ is $2^k$.
\end{corollary}

\begin{proof}
The upper bound follows from Corollary \ref{cor_bound2}. The bound is achieved by the graph $G = B(K_k, K_{2^k})$, which has $\adim(G) = \bdim(G) = k$.
\end{proof}

The next result is sharp up to a constant factor, as shown by paths, cycles, and grid graphs.

\begin{proposition}\label{bdim_diam}
For graphs $G$ of diameter $d$, $\adim(G) \geq \bdim(G) \geq \frac{d}{3}$. 
\end{proposition}

\begin{proof}
Suppose that $G$ is a graph of diameter $d$ with minimal path $v_1, \dots, v_{d+1}$ between two vertices $v_1, v_{d+1}$ with distance $d$, and let $f$ be a resolving broadcast that achieves $\bdim(G)$. Then $|\supp_G(f)|+2\sum_{v \in V(G)} f(v) \geq d$, which implies that $\bdim(G) \geq \frac{d}{3}$.~\hfill
\end{proof}

Thus we have a sharp bound on $\bdim(G)$ up to a constant factor for any graph $G$ with $\dim(G) = O(1)$, where the upper bound follows from the definition of $\bdim(G)$.

\begin{theorem}\label{bdim_diam}
For every graph $G$ of diameter $d$, $\frac{d}{3} \leq \bdim(G) \leq \dim(G) (d-1)$.
\end{theorem}

\begin{corollary}
If $G$ has diameter $d$ and $\dim(G) = O(1)$, then $\bdim(G) = \Theta(d)$. If $G$ has diameter $d = O(1)$, then $\bdim(G) = \Theta(\dim(G))$.
\end{corollary}

The next result is sharp up to a constant factor by Theorem \ref{O_sharp}.

\begin{theorem}\label{thm_log}
For all graphs $G$ of order $n$, $\adim(G) \geq \bdim(G) = \Omega(\log n)$.
\end{theorem}

\begin{proof}
Suppose that $f$ is a resolving broadcast that achieves $\bdim(G)$, and let $y = |\supp_G(f)|$. The $\Omega(\log n)$ bound holds if $y > \ln (\frac{n}{2})$, so we suppose that $y \leq  \ln (\frac{n}{2})$. Since $f$ is a resolving broadcast for $G$, we must have $y+\prod_{v \in \supp_G(f)} (f(v)+1) \geq n$, which implies by the arithmetic-geometric mean inequality that $y+\sum_{v \in V(G)} f(v) \geq y (n-y)^{1/y}$, or equivalently $\sum_{v \in V(G)} f(v) \geq y (n-y)^{1/y}-y$. Since $y (n-y)^{1/y} \geq y(\frac{n}{2})^{1/y} \geq y e$ for $n$ sufficiently large, we have $\sum_{v \in V(G)} f(v) \geq \frac{e-1}{e} y(\frac{n}{2})^{1/y}$. 

Define $g(y) = \ln(\frac{e-1}{e} y(\frac{n}{2})^{1/y})$, so $g'(y) = \frac{1}{y}-\frac{\ln (\frac{n}{2})}{y^2}$, which has one root at $y = \ln (\frac{n}{2})$. This is a minimum since $g'(y) < 0$ for $y < \ln (\frac{n}{2})$ and $g'(y) > 0$ for $y > \ln (\frac{n}{2})$. Since $\ln(x)$ is an increasing function, $\frac{e-1}{e} y(\frac{n}{2})^{1/y}$ is also minimized at $y = \ln (\frac{n}{2})$, where it has value $(e-1) \ln (\frac{n}{2})$. Thus $\sum_{v \in V(G)} f(v) \geq (e-1) \ln (\frac{n}{2})$ in this case.~\hfill
\end{proof}

\section{Graphs $G$ having $\bdim(G)$ equal to $1$, $2$, and $|V(G)|-1$}\label{s:highlow}

Next, we characterize graphs $G$ having $\bdim(G)$ equal to $1$, $2$, and $|V(G)|-1$. We begin with the following known results on metric dimension and adjacency dimension.

\begin{theorem}\emph{\cite{tree1}}\label{dim_characterization}
Let $G$ be a connected graph of order $n$. Then 
\begin{itemize}
\item[(a)] $\dim(G)=1$ if and only if $G=P_n$;
\item[(b)] for $n \ge 4$, $\dim(G)=n-2$ if and only if $G=K_{s,t}$ ($s,t \ge 1$), $G=K_s+\overline{K}_t$ ($s\ge1, t\ge2$), or $G=K_s+(K_1 \cup K_t)$ ($s,t \ge 1$); 
\item[(c)] $\dim(G)=n-1$ if and only if $G=K_n$.
\end{itemize}
\end{theorem}

\begin{theorem}\emph{\cite{adjacency_dim}}\label{adj_characterization}
Let $G$ be a graph of order $n$. Then 
\begin{itemize}
\item[(a)] $\adim(G)=1$ if and only if $G \in \{P_1, P_2, P_3, \overline{P}_2, \overline{P}_3\}$;
\item[(b)] $\adim(G)=n-1$ if and only if $G\in\{K_n, \overline{K}_n\}$.
\end{itemize}
\end{theorem}

Note that, if $f$ is a resolving broadcast of $G$ with $f(v)=2$ and $f(w)=0$ for each $w\in V(G)-\{v\}$, then $v$ is an end vertex of $P_4$ or $v$ is an end vertex of $P_3 \cup P_1$, and $\adim(P_4)=\adim(P_3 \cup P_1)=2$ as shown in Theorem~\ref{adimch}. Also, note that $\adim(G)=2$ implies $\bdim(G)=2$. So, Observation~\ref{obs_bdim}(a), Theorems~\ref{adimch},~\ref{dim_characterization} and~\ref{adj_characterization} imply the following proposition.

\begin{proposition}\label{bdim_characterization}
Let $G$ be a graph of order $n$. Then 
\begin{itemize}
\item[(a)] $\bdim(G)=1$ if and only if $G \in \{P_1, P_2, P_3, \overline{P}_2, \overline{P}_3\}$;
\item[(b)] $\bdim(G)=2$ if and only if $G\in\mathcal{H}_2-\mathcal{H}_1$ as described in Theorem~\ref{adimch} (see Figure~\ref{fig_adim2}); 
\item[(c)] $\bdim(G)=n-1$ if and only if $G\in\{K_n, \overline{K}_n\}$.
\end{itemize}
\end{proposition}

The next two questions about graphs with high adjacency dimension and broadcast dimension are both open.

\begin{question}What graphs $G$ of order $n$ satisfy $\adim(G)=n-2$? \end{question}

\begin{question}What graphs $G$ of order $n$ satisfy $\bdim(G)=n-2$? \end{question}

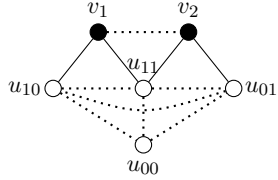
\begin{figure}[ht]
\centering
\begin{tikzpicture}[scale=.75, transform shape]

\node [draw, fill=black, shape=circle, scale=.8] (u) at  (-0.3, 1) {};
\node [draw, fill=black, shape=circle, scale=.8] (v) at  (1.3, 1) {};
\node [draw, shape=circle, scale=.8] (0) at  (-1.1, 0) {};
\node [draw, shape=circle, scale=.8] (1) at  (0.5, 0) {};
\node [draw, shape=circle, scale=.8] (2) at  (2.1, 0) {};
\node [draw, shape=circle, scale=.8] (3) at  (0.5, -1) {};

\node [scale=1.1] at (-0.3,1.4) {$v_1$};
\node [scale=1.1] at (1.3,1.4) {$v_2$};

\node [scale=1.1] at (0.5, 0.4) {$u_{11}$};
\node [scale=1.1] at (-1.6,0) {$u_{10}$};
\node [scale=1.1] at (2.6,0) {$u_{01}$};
\node [scale=1.1] at (0.5,-1.4) {$u_{00}$};

\draw(0)--(u)--(1)--(v)--(2);
\draw[thick,dotted](u)--(v);
\draw[thick,dotted](0)--(1)--(2).. controls (0.5,-0.5) .. (0);\draw[thick,dotted](0)--(3)--(2);\draw[thick, dotted](1)--(3);

\end{tikzpicture}
\caption{\small The graphs $G$ satisfying $\adim(G)=2$, where black vertices must be present, a solid edge must be present whenever the two vertices incident to the solid edge are in the graph, but a dotted edge is not necessarily present.}\label{fig_adim2}
\end{figure}

A graph is \emph{planar} if it can be drawn in a plane without edge crossing. For two graphs $G$ and $H$, $H$ is called a \emph{minor} of $G$ if $H$ can be obtained from $G$ by vertex deletion, edge deletion, or edge contraction.

\begin{theorem}\emph{\cite{wagner}}\label{minor}
A graph $G$ is planar if and only if neither $K_5$ nor $K_{3,3}$ is a minor of $G$.
\end{theorem}

\begin{remark}\label{bdim2_planar}
It was shown in~\cite{tree2} that there exists a non-planar graph $G$ with $\dim(G)=2$. However, $\adim(G)=2$ ($\bdim(G)=2$, respectively) implies $G$ is planar (see Figure~\ref{fig_adim2}). Also, note that, for each $k \ge 3$, there exists a non-planar graph $G$ satisfying $\bdim(G)=k$ and $\adim(G)=k$, respectively. For example, the graph $G$ of order $n=k+2^k$ with $\bdim(G)=k$ ($\adim(G)=k$, respectively) described in the proof of Theorem~\ref{O_sharp} contains $K_{2^k}$ as a subgraph. Since $K_{2^k}$, for $k \ge 3$, contains $K_5$ as a minor, $G$ is not planar by Theorem~\ref{minor}.
\end{remark}

\section{Comparing $\dim(G)$, $\adim(G)$, and $\bdim(G)$}\label{s:comparing}

Next, we provide a connected graph $G$ such that both $\adim(G)-\bdim(G)$ and $\bdim(G)-\dim(G)$ can be arbitrarily large. In fact, we obtain the stronger result that $\frac{\adim(G)}{\bdim(G)}$ and $\frac{\bdim(G)}{\dim(G)}$ can be arbitrarily large. We first recall some results on grid graphs.

\begin{proposition}\emph{\cite{cartesian_product}}\label{dim_grid}
For the grid graph $G=P_m \times P_n$ ($m,n \ge 2$), $\dim(G)=2$. 
\end{proposition}

\begin{proposition}\emph{\cite{tree2}}\label{d-grid}
For the $d$-dimensional grid graph $G=\prod_{i=1}^{d}P_{n_i}$, where $d \ge 2$ and $n_i \ge 2$ for each $i \in \{1,\ldots,d\}$, $\dim(G) \le d$.
\end{proposition}

With Theorem \ref{bdim_diam}, propositions~\ref{dim_grid} and~\ref{d-grid} immediately imply the following corollary.

\begin{corollary}\label{grid}
If $G$ is the grid graph $P_m \times P_n$ ($m,n\ge 2$), then $\bdim(G) = \Theta(m+n)$. More generally, if $G$ is the $d$-dimensional grid graph $\prod_{i = 1}^{d} P_{n_i}$ with $n_i \ge 2$ for each $i=1, \dots, d$, then $\bdim(G) = \Theta(\sum_{i = 1}^{d} n_i)$, where the constant in the upper bound depends on $d \ge 2$.
\end{corollary}

\begin{theorem}\label{comparison_3dim}
For $k \geq 2$, let $G$ be the $d$-dimensional grid graph $\prod_{i = 1}^{d} P_k$. Then $\bdim(G) = \Theta(k)$, and $\adim(G)=\Theta(k^d)$, where the constants in the bounds depend on $d$. So, $\frac{\adim(G)}{\bdim(G)}$ and $\frac{\bdim(G)}{\dim(G)}$ can be arbitrarily large.
\end{theorem}

\begin{proof}
Note that $\dim(G) \leq d$ by Proposition~\ref{d-grid} and $\bdim(G) = \Theta(k)$ by Corollary~\ref{grid}.
To see that $\adim(G)=\Theta(k^d)$, first note that $\adim(G)=O(k^d)$ since $|V(G)| = O(k^d)$. Moreover, any adjacency resolving set of $G$ must contain at least one vertex from every $\prod_{i = 1}^{d} P_3$ subgraph of $G$ except for at most one, so $\adim(G) = \Omega(k^d)$.~\hfill
\end{proof}

In the next result, we show that the multiplicative gap between $\bdim(G)$ and $\adim(G)$ in Theorem \ref{comparison_3dim} is tight up to a constant factor. To state this result, we define $\Delta'(G)$ to be the maximum value of $t$ for which there exists a positive integer $j$ and a vertex $v \in V(G)$ such that there exist at least $t$ distinct vertices $u_1, \dots, u_t \in V(G)$ with $d_G(u_i, v) = j$ for each $i = 1, \dots, t$. Note that when $G = P_k \times P_k$, we have $\Delta'(G) = \Theta(k)$, so $\frac{\adim(G)}{\bdim(G)} = \Theta(k) = \Theta(\Delta'(G))$.

\begin{proposition}
For all graphs $G$, $\frac{\adim(G)}{\bdim(G)} = O(\Delta'(G))$.
\end{proposition}

\begin{proof}
Given a resolving broadcast $f$ of $G$ with $\sum_{v \in V(G)} f(v) = \bdim(G)$, we show how to convert $f$ into an adjacency resolving set for $G$ which uses at most $(\Delta'(G)+1) \bdim(G)$ vertices. Let $v$ be a vertex $v \in V(G)$ with $f(v) > 0$. If $f(v) = 1$, then we put the vertex $v$ into the adjacency resolving set for $G$. If $f(v) > 1$, then for each $k > 0$, we list the vertices $u_1, \dots, u_t$ with $d_G(u_i, v) = k$, and we add each vertex $u_1, \dots, u_t$ into the adjacency resolving set for $G$, as well as the vertex $v$. Thus we add at most $(\Delta'(G)+1)$ vertices to the adjacency resolving set for each vertex $v \in V(G)$ with $f(v) > 0$ and each positive integer $k$ with $k \leq f(v)$. This implies that $\adim(G) \leq (\Delta'(G)+1) \bdim(G)$.
\end{proof}

The proof of the last proposition also implies the following proposition.

\begin{proposition}
For all graphs $G$ of order $n$, $\bdim(G)\Delta'(G) = \Omega(n)$.
\end{proposition}

It was shown in~\cite{linegraph} that metric dimension is not a monotone parameter on subgraph inclusion; see~\cite{linegraph} for an example satisfying $H \subset G$ and $\dim(H) > \dim(G)$. Next, we show that for two graphs $G$ and $H$ with $H \subset G$, $\dim(H)-\dim(G)$, $\bdim(H)-\bdim(G)$, and $\adim(H)-\adim(G)$ can be arbitrarily large. In fact, we obtain the stronger result that $\frac{\dim(H)}{\dim(G)}$, $\frac{\bdim(H)}{\bdim(G)}$, and $\frac{\adim(H)}{\adim(G)}$ can be arbitrarily large.

\begin{theorem}\label{remark_difference}
There exist connected graphs $G$ and $H$ such that $H \subset G$ and $\frac{\dim(H)}{\dim(G)}$, $\frac{\bdim(H)}{\bdim(G)}$, and $\frac{\adim(H)}{\adim(G)}$ can be arbitrarily large.
\end{theorem}

\begin{proof}
For $k \ge 3$, let $H=K_{\frac{k(k+1)}{2}}$; let $V(H)$ be partitioned into $V_1, V_2, \ldots, V_k$ such that $V_i=\{w_{i,1}, w_{i,2}, \ldots, w_{i,i}\}$ with $|V_i|=i$, where $i \in \{1,2,\ldots,k\}$. Let $G$ be the graph obtained from $H$ and $k$ isolated vertices $u_1,u_2, \ldots, u_k$ as follows: $u_1$ is adjacent to $V_1 \cup (\cup_{j=2}^{k}\{w_{j,1}\})$, $u_2$ is adjacent to $V_2 \cup (\cup_{j=3}^{k}\{w_{j,2}\})$, $u_3$ is adjacent to $V_3 \cup (\cup_{j=4}^{k}\{w_{j,3}\})$, and so on, i.e., for each $i \in \{1,2,\ldots, k\}$, $u_i$ is adjacent to each vertex of $V_i \cup (\cup_{j=i+1}^{k}\{w_{j,i}\})$ (see the graph $G$ in Figure~\ref{fig_remark2} when $k=4$). Since $\diam(H)=1$ and $\diam(G)=2$, $\dim(H)=\bdim(H)=\adim(H)$ and $\dim(G)=\bdim(G)=\adim(G)$ by Observation~\ref{obs_bdim}(b). Note that $H\subset G$ and $\dim(H) =\frac{k(k+1)}{2}-1$ by Theorem~\ref{dim_characterization}(c). Since $\{u_1, u_2, \ldots, u_k\}$ forms a resolving set of $G$, $\dim(G) \le k$. So, $\frac{\dim(H)}{\dim(G)}=\frac{\bdim(H)}{\bdim(G)}=\frac{\adim(H)}{\adim(G)}\ge \frac{k^2+k-2}{2k} \rightarrow \infty$ as $k \rightarrow \infty$.~\hfill  
\end{proof}

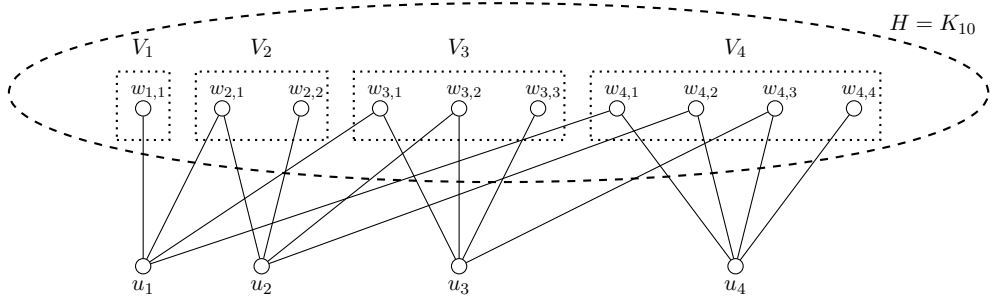
\begin{figure}[ht]
\centering
\begin{tikzpicture}[scale=.7, transform shape]

\node [draw, shape=circle, scale=.8] (0) at  (0, 0) {};
\node [draw, shape=circle, scale=.8] (1) at  (1.5, 0) {};
\node [draw, shape=circle, scale=.8] (2) at  (3, 0) {};
\node [draw, shape=circle, scale=.8] (3) at  (4.5, 0) {};
\node [draw, shape=circle, scale=.8] (4) at  (6, 0) {};
\node [draw, shape=circle, scale=.8] (5) at  (7.5, 0) {};
\node [draw, shape=circle, scale=.8] (6) at  (9, 0) {};
\node [draw, shape=circle, scale=.8] (7) at  (10.5, 0) {};
\node [draw, shape=circle, scale=.8] (8) at  (12, 0) {};
\node [draw, shape=circle, scale=.8] (9) at  (13.5, 0) {};

\node [draw, shape=circle, scale=.8] (a) at  (0, -3) {};
\node [draw, shape=circle, scale=.8] (b) at  (2.25, -3) {};
\node [draw, shape=circle, scale=.8] (c) at  (6, -3) {};
\node [draw, shape=circle, scale=.8] (d) at  (11.25, -3) {};

\draw[thick, dotted](-0.5,-0.6) rectangle (0.5,0.7);
\draw[thick, dotted](1,-0.6) rectangle (3.5,0.7);
\draw[thick, dotted](4,-0.6) rectangle (8,0.7);
\draw[thick, dotted](8.5,-0.6) rectangle (14,0.7);

\node [scale=1] at (0.1,0.3) {$w_{1,1}$};
\node [scale=1] at (1.6,0.3) {$w_{2,1}$};
\node [scale=1] at (3.1,0.3) {$w_{2,2}$};
\node [scale=1] at (4.6,0.3) {$w_{3,1}$};
\node [scale=1] at (6.1,0.3) {$w_{3,2}$};
\node [scale=1] at (7.6,0.3) {$w_{3,3}$};
\node [scale=1] at (9.1,0.3) {$w_{4,1}$};
\node [scale=1] at (10.6,0.3) {$w_{4,2}$};
\node [scale=1] at (12.1,0.3) {$w_{4,3}$};
\node [scale=1] at (13.6,0.3) {$w_{4,4}$};

\node [scale=1.1] at (0,-3.4) {$u_{1}$};
\node [scale=1.1] at (2.25,-3.4) {$u_{2}$};
\node [scale=1.1] at (6,-3.4) {$u_{3}$};
\node [scale=1.1] at (11.25,-3.4) {$u_{4}$};

\node [scale=1.1] at (0,1.15) {$V_1$};
\node [scale=1.1] at (2.25,1.15) {$V_2$};
\node [scale=1.1] at (6,1.15) {$V_3$};
\node [scale=1.1] at (11.25,1.15) {$V_4$};
\node [scale=1.1] at (15,1.6) {$H=K_{10}$};

\draw[thick,dashed] (6.75,0.3) ellipse (9.3cm and 1.7cm);

\draw(0)--(a)--(1);\draw(3)--(a)--(6);\draw(1)--(b)--(2);\draw(4)--(b)--(7);\draw(3)--(c)--(4);\draw(5)--(c)--(8);\draw(6)--(d)--(7);\draw(8)--(d)--(9);
\end{tikzpicture}
\caption{\small A graph $G$ such that $H \subset G$ and $\frac{\dim(H)}{\dim(G)}=\frac{\bdim(H)}{\bdim(G)}=\frac{\adim(H)}{\adim(G)}$ can be arbitrarily large; here, $k=4$ and $H=K_{10}$ for the example described in Theorem~\ref{remark_difference}.}\label{fig_remark2}
\end{figure}

Next we find all trees $T$ for which $\dim(T)=\bdim(T)$. First we recall some terminology. Fix a tree $T$. An end vertex $\ell$ is called a \emph{terminal vertex} of a major vertex $v$ if $d(\ell, v)<d (\ell, w)$ for every other major vertex $w$ in $T$. The \emph{terminal degree}, $ter(v)$, of a major vertex $v$ is the number of terminal vertices of $v$ in $T$, and an \emph{exterior major vertex} is a major vertex that has positive terminal degree. We denote by $ex(T)$ the number of exterior major vertices of $T$, and $\sigma(T)$ the number of end vertices of $T$. 

\begin{theorem}\emph{\cite{tree1, tree2, tree3}}\label{dim_tree}
For a tree $T$ that is not a path, $\dim(T)=\sigma(T)-ex(T)$.
\end{theorem}

\begin{theorem}\emph{\cite{tree3}}\label{zhang_tree}
Let $T$ be a tree with $ex(T)=k \ge 1$, and let $v_1, v_2, \ldots, v_k$ be the exterior major vertices of $T$. For each $i$ ($1 \le i \le k$), let $\ell_{i,1}, \ell_{i,2}, \ldots, \ell_{i, \sigma_i}$ be the terminal vertices of $v_i$ with $ter(v_i)=\sigma_i \ge 1$, and let $P_{i,j}$ be the $v_i-\ell_{i,j}$ path, where $1 \le j \le \sigma_i$. Let $W \subseteq V(T)$. Then $W$ is a minimum resolving set of $T$ if and only if $W$ contains exactly one vertex from each of the paths $P_{i,j}-v_i$ ($1 \le j \le \sigma_i$ and $1 \le i\le k$) with exactly one exception for each $i$ with $1 \le i \le k$ and $W$ contains no other vertices of $T$.
\end{theorem}

\begin{proposition}\label{tree_bdim}
Let $T$ be a non-trivial tree. Then $\dim(T)=\bdim(T)$ if and only if $T\in\{P_2, P_3\}$ or $T$ is a tree obtained from the star $K_{1,x}$ ($x \ge 3$) by subdividing at most $x-1$ edges exactly once.
\end{proposition}

\begin{proof}
($\Leftarrow$) First, let $T\in\{P_2, P_3\}$, and let $\ell$ be an end vertex of $T$. Let $g$ be a function defined on $V(G)$ such that $g(\ell)=1$ and $g(v)=0$ for each $v\in V(T)-\{\ell\}$. Then $g$ is a resolving broadcast of $T$, and thus $\bdim(T)=1=\dim(T)$ by Observation~\ref{obs_bdim}(a) and Theorem~\ref{dim_characterization}(a). Second, let $T$ be a tree obtained from the star $K_{1,x}$ ($x \ge 3$) by subdividing at most $x-1$ edges exactly once. Let $w$ be the major vertex of $T$, and let $\ell_1, \ell_2, \ldots, \ell_x$ be the terminal vertices of $w$ in $T$ such that $d(w, \ell_1) \ge d(w, \ell_2) \ge \ldots \ge d(w, \ell_x)$; then $d(w, \ell_x)=1$. If $f:V(T) \rightarrow \mathbb{Z}^+ \cup \{0\}$ is a function defined by 
\begin{equation*}
f(v)=\left\{
\begin{array}{ll}
1 & \mbox{ if } v \in N(w)-\{\ell_x\}\\
0 & \mbox{ otherwise,}
\end{array}
\right.
\end{equation*}
then $f$ is a resolving broadcast of $T$, and thus $\bdim(T) \le x-1=\dim(T)$ by Theorem~\ref{dim_tree}. By Observation~\ref{obs_bdim}(a), $\bdim(T)=\dim(T)$.

($\Rightarrow$) Let $\dim(T)=\bdim(T)$. Let $f: V(T) \rightarrow \mathbb{Z}^+ \cup \{0\}$ be a resolving broadcast of $T$ with $c_f(T)=\dim(T)$, and let $R=\supp_T(f)$. First, let $ex(T)=0$, i.e., $T$ is a path; then $c_f(T)=1$ by Theorem~\ref{dim_characterization}(a). So, $b_R(u)\in\{0,1,2\}$ for each $u\in V(T)$. Thus, $T\in\{P_2, P_3\}$. Second, let $ex(T)=1$. Let $v$ be the exterior major vertex of $T$ with terminal vertices $\ell_1, \ell_2, \ldots, \ell_x$ such that $d(v, \ell_1) \ge d(v, \ell_2) \ge \ldots \ge d(v, \ell_x)$; then $x \ge 3$ and $\dim(T)=x-1$ by Theorem~\ref{dim_tree}. Further, let $N(v)=\cup_{i=1}^{x}\{s_i\}$ and let $s_i$ lie on the $v-\ell_i$ path, where $i\in\{1,2,\ldots, x\}$. If $d(v, \ell_x) \ge 2$, then there exists $j\in \{1,2,\ldots, x\}$ such that $d(v, \ell_j) \ge 2$ and $b_R(s_j)=b_R(\ell_j)$, contradicting the assumption that $f$ is a resolving broadcast of $T$. So, $d(v, \ell_x)=1$ and $b_R(\ell_x)={\bf 2}_{x-1}$. If $d(v, \ell_1)=d \ge 3$, say the $v-\ell_1$ path is given by $v, s_1, s_2, \ldots, s_d=\ell_1$, then (i) $b_R(\ell_1)=b_R(\ell_x)$ if $f(s_1)=1$; (ii) $b_R(s_{i-1})=b_R(s_{i+1})$ if $f(s_i)=1$ for some $i\in\{2,3,\ldots, d-1\}$; (iii) $b_R(s_{1})=b_R(\ell_x)$ if $f(\ell_1)=1$. So, $d(v,\ell_1) \le 2$. Next, let $ex(T) \ge 2$; we show that $\bdim(T)>\dim(T)$. Let $v_1, v_2, \ldots, v_a$ be distinct exterior major vertices of $T$, where $a\ge 2$. Let $\ell_1, \ell_2, \ldots, \ell_{\alpha}$ be the set of terminal vertices of $v_1$, and let $\ell'_1, \ell'_2, \ldots, \ell'_{\beta}$ be the set of terminal vertices of $v_2$ in $T$; let $P^{1,i}$ be the $v_1-\ell_i$ path excluding $v_1$, and let $P^{2,j}$ be the $v_2-\ell'_j$ path excluding $v_2$ in $T$. By Theorem~\ref{zhang_tree}, $c_f(P^{1,x})=c_f(P^{2,y})=0$ for some $x\in\{1,2,\ldots, \alpha\}$ and $y\in\{1,2\ldots, \beta\}$. Since $b_R(\ell_{1,x})=b_R(\ell'_{2,y})={\bf 2}_{|R|}$, $f$ fails to be a resolving broadcast of $T$, and thus $\bdim(T)>\dim(T)$.
\end{proof}

Proposition~\ref{tree_bdim} implies the following corollary.

\begin{corollary}\label{tree_b_cor}
For any non-trivial tree $T$, $\dim(T)=\adim(T)$ if and only if $T\in\{P_2, P_3\}$ or $T$ is a tree obtained from the star $K_{1,x}$ ($x \ge 3$) by subdividing at most $x-1$ edges exactly once.
\end{corollary}

\section{The effect of vertex or edge deletion on the adjacency dimension and the broadcast dimension of graphs}\label{s:delete}

Throughout this section, let $v$ and $e$, respectively, denote a vertex and an edge of a connected graph $G$ such that both $G-v$ and $G-e$ are connected graphs. First, we consider the effect of vertex deletion on adjacency dimension and broadcast dimension. It is known that $\dim(G)-\dim(G-v)$ and $\dim(G-v)-\dim(G)$, respectively, can be arbitrarily large; see~\cite{wheel1} and~\cite{joc}, respectively. We show that $\frac{\bdim(G)}{\bdim(G-v)}$ can be arbitrarily large, whereas $\adim(G) \le \adim(G-v)+1$. We also show that $\bdim(G-v) - \bdim(G)$ and $\adim(G-v)-\adim(G)$ can be arbitrarily large.

We recall the following useful result.

\begin{proposition}\emph{\cite{k_adjacency_dim}}\label{adim_uni}
Let $H$ be a graph of order $n \geq 2$. Then $\adim(K_1 + H) \geq \adim(H)$. 
\end{proposition}

\begin{remark}
The value of $\frac{\bdim(G)}{\bdim(G-v)}$ can be arbitrarily large, as $G$ varies.
\end{remark}

\begin{proof}
Let $G = (P_k \times P_k) + K_1$, and let $v$ be the vertex in the $K_1$. Then $\bdim(G-v) = \Theta(k)$ by Theorem \ref{comparison_3dim}, but $\bdim(G) = \adim(G) \geq \adim(G-v) = \Theta(k^2)$ by Proposition~\ref{adim_uni} and Theorem~\ref{comparison_3dim}.~\hfill
\end{proof}

\begin{proposition}
For any graph $G$, $\adim(G) \le \adim(G-v)+1$, where the bound is sharp.
\end{proposition}

\begin{proof}
Let $S$ be a minimum adjacency resolving set of $G-v$. Note that, for any vertex $x$ in $G-v$, $a_S(x)$ in $G-v$ remains the same in $G$. So, $S\cup\{v\}$ forms an adjacency resolving set of $G$, and hence $\adim(G) \le |S|+1=\adim(G-v)+1$. For the sharpness of the bound, let $G=K_n$ for $n \ge 3$; then $\adim(G)=n-1$ and $\adim(G-v)=n-2$, for any $v\in V(G)$, by Theorem~\ref{adj_characterization}(b).~\hfill  
\end{proof}

\begin{remark}
The value of $\bdim(G-v) - \bdim(G)$ and $\adim(G-v)-\adim(G)$ can be arbitrarily large, as $G$ varies.
\end{remark}

\begin{proof}
Let $G$ be the graph in Figure~\ref{fig_vdeletion}, where $k \ge 2$. Note that $\diam(G)=\diam(G-v)=2$; thus, $\dim(G)=\bdim(G)=\adim(G)$ and $\dim(G-v)=\bdim(G-v)=\adim(G-v)$ by Observation~\ref{obs_bdim}(b).

First, we show that $\dim(G)=k+1$. Let $S$ be any minimum resolving set of $G$. Note that, for each $i\in\{1,2,\ldots, k\}$, $x_i$ and $z_i$ are twin vertices of $G$; thus $|S \cap \{x_i, z_i\}| \ge 1$ by Observation~\ref{obs_twin}(a). Without loss of generality, let $S'=\cup_{i=1}^{k}\{x_i\} \subseteq S$. Since $r_{S'}(y_i)=r_{S'}(z_i)$ for each $i \in \{1,2,\ldots, k\}$, $|S| \ge k+1$; thus $\dim(G) \ge k+1$. On the other hand, $S'\cup \{v\}$ forms a resolving set of $G$, and thus $\dim(G) \le k+1$. So, $\dim(G)=k+1$.

Second, we show that $\dim(G-v)=2k$. Let $R$ be any minimum resolving set of $G-v$. Note that, for each $i\in\{1,2,\ldots, k\}$, any two vertices in $\{x_i, y_i, z_i\}$ are twin vertices of $G-v$. By Observation~\ref{obs_twin}(a), $|R \cap\{x_i, y_i, z_i\}| \ge 2$ for each $i\in\{1,2,\ldots, k\}$; thus $|R| \ge 2k$. Since $\cup_{i=1}^{k}\{x_i, y_i\}$ forms a resolving set of $G-v$, $\dim(G-v) \le 2k$. Thus, $\dim(G-v)=2k$.

Therefore, $\dim(G-v) - \dim(G)=\bdim(G-v) - \bdim(G)=\adim(G-v)-\adim(G)=2k-(k+1)=k-1 \rightarrow \infty$ as $k \rightarrow \infty$.~\hfill
\end{proof}

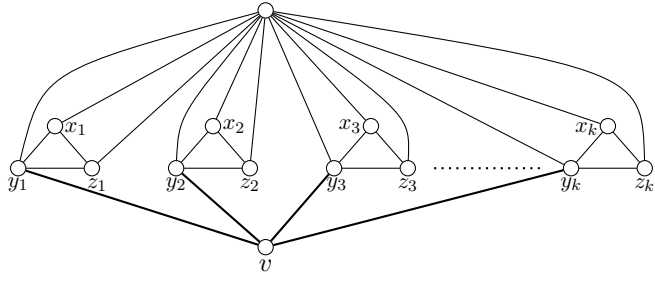
\begin{figure}[ht]
\centering
\begin{tikzpicture}[scale=.7, transform shape]

\node [draw, shape=circle, scale=.8] (0) at  (0, 3) {};
\node [draw, shape=circle, scale=.8] (1) at  (-4.7, 0) {};
\node [draw, shape=circle, scale=.8] (2) at  (-4, 0.8) {};
\node [draw, shape=circle, scale=.8] (3) at  (-3.3, 0) {};
\node [draw, shape=circle, scale=.8] (4) at  (-1.7, 0) {};
\node [draw, shape=circle, scale=.8] (5) at  (-1, 0.8) {};
\node [draw, shape=circle, scale=.8] (6) at  (-0.3, 0) {};
\node [draw, shape=circle, scale=.8] (7) at  (1.3, 0) {};
\node [draw, shape=circle, scale=.8] (8) at  (2, 0.8) {};
\node [draw, shape=circle, scale=.8] (9) at  (2.7, 0) {};
\node [draw, shape=circle, scale=.8] (10) at  (5.8, 0) {};
\node [draw, shape=circle, scale=.8] (11) at  (6.5, 0.8) {};
\node [draw, shape=circle, scale=.8] (12) at  (7.2, 0) {};
\node [draw, shape=circle, scale=.8] (v) at  (0, -1.5) {};

\node [scale=1.3] at (0,-1.85) {$v$};
\node [scale=1.1] at (-3.6,0.75) {$x_1$};
\node [scale=1.1] at (-4.7,-0.3) {$y_1$};
\node [scale=1.1] at (-3.2,-0.31) {$z_1$};
\node [scale=1.1] at (-0.6,0.8) {$x_2$};
\node [scale=1.1] at (-1.7,-0.3) {$y_2$};
\node [scale=1.1] at (-0.3,-0.31) {$z_2$};
\node [scale=1.1] at (1.6,0.8) {$x_3$};
\node [scale=1.1] at (1.35,-0.3) {$y_3$};
\node [scale=1.1] at (2.7,-0.31) {$z_3$};
\node [scale=1.1] at (6.1,0.75) {$x_k$};
\node [scale=1.1] at (5.8,-0.3) {$y_k$};
\node [scale=1.1] at (7.2,-0.31) {$z_k$};

\draw[thick,dotted] (3.2,0)--(5.3,0);

\draw(1)--(2)--(3)--(1);\draw(4)--(5)--(6)--(4);\draw(7)--(8)--(9)--(7);\draw(10)--(11)--(12)--(10);\draw(2)--(0);\draw(5)--(0);\draw(8)--(0);\draw(11)--(0);\draw(6)--(0);\draw(7)--(0);\draw(3)--(0);\draw(10)--(0);
\draw(1) .. controls(-4.25,1.85) .. (0);\draw(4) .. controls(-1.65,1) .. (0);\draw(9) .. controls(2.75,1) .. (0);\draw(12) .. controls(7.15,1.85) .. (0);

\draw[thick](v)--(1);\draw[thick](v)--(4);\draw[thick](v)--(7);\draw[thick](v)--(10);

\end{tikzpicture}
\caption{\small A graph $G$ such that $\dim(G-v) - \dim(G)=\bdim(G-v) - \bdim(G)=\adim(G-v)-\adim(G)$ can be arbitrarily large, where $k \ge 2$.}\label{fig_vdeletion}
\end{figure}

Next, we consider the effect of edge deletion. We recall the following result on the effect of edge deletion on metric dimension.

\begin{theorem}\emph{\cite{joc}}
\begin{itemize}
\item[(a)] For any graph $G$ and any edge $e \in E(G)$, $\dim(G-e) \le \dim(G)+2$. 
\item[(b)] The value of $\dim(G) - \dim(G - e)$ can be arbitrarily large (see Figure~\ref{fig_dim_edge}). 
\end{itemize}
\end{theorem}

\begin{figure}[ht]
\centering
\begin{tikzpicture}[scale=.7, transform shape]

\node [draw, shape=circle, scale=.8] (0) at  (0, -1.5) {};
\node [draw, shape=circle, scale=.8] (1) at  (1, 0) {};
\node [draw, shape=circle, scale=.8] (2) at  (2.1, 0.4) {};
\node [draw, shape=circle, scale=.8] (3) at  (3.3, 0.4) {};
\node [draw, shape=circle, scale=.8] (4) at  (2, -0.4) {};
\node [draw, shape=circle, scale=.8] (5) at  (2.7, -0.4) {};
\node [draw, shape=circle, scale=.8] (6) at  (3.4, -0.4) {};
\node [draw, shape=circle, scale=.8] (7) at  (4.5, 0) {};
\node [draw, shape=circle, scale=.8] (8) at  (5.5, 0) {};
\node [draw, shape=circle, scale=.8] (9) at  (6.5, 0.4) {};
\node [draw, shape=circle, scale=.8] (10) at  (6.5, -0.4) {};
\node [draw, shape=circle, scale=.8] (11) at  (7.5, -1.5) {};

\node [draw, shape=circle, scale=.8] (a1) at  (1, -1.5) {};
\node [draw, shape=circle, scale=.8] (a2) at  (2.1, -1.1) {};
\node [draw, shape=circle, scale=.8] (a3) at  (3.3, -1.1) {};
\node [draw, shape=circle, scale=.8] (a4) at  (2, -1.9) {};
\node [draw, shape=circle, scale=.8] (a5) at  (2.7, -1.9) {};
\node [draw, shape=circle, scale=.8] (a6) at  (3.4, -1.9) {};
\node [draw, shape=circle, scale=.8] (a7) at  (4.5, -1.5) {};
\node [draw, shape=circle, scale=.8] (a8) at  (5.5, -1.5) {};
\node [draw, shape=circle, scale=.8] (a9) at  (6.5, -1.1) {};
\node [draw, shape=circle, scale=.8] (a10) at  (6.5, -1.9) {};

\node [draw, shape=circle, scale=.8] (b1) at  (1, -4) {};
\node [draw, shape=circle, scale=.8] (b2) at  (2.1, -3.6) {};
\node [draw, shape=circle, scale=.8] (b3) at  (3.3, -3.6) {};
\node [draw, shape=circle, scale=.8] (b4) at  (2, -4.4) {};
\node [draw, shape=circle, scale=.8] (b5) at  (2.7, -4.4) {};
\node [draw, shape=circle, scale=.8] (b6) at  (3.4, -4.4) {};
\node [draw, shape=circle, scale=.8] (b7) at  (4.5, -4) {};
\node [draw, shape=circle, scale=.8] (b8) at  (5.5, -4) {};
\node [draw, shape=circle, scale=.8] (b9) at  (6.5, -3.6) {};
\node [draw, shape=circle, scale=.8] (b10) at  (6.5, -4.4) {};

\node [scale=1.1] at (1.1,-0.35) {$u_1$};
\node [scale=1.1] at (1.1,-1.85) {$u_2$};
\node [scale=1.1] at (1.1,-4.35) {$u_k$};
\node [scale=1.3] at (3.75,2.1) {\bf$e$};

\draw[thick](7.5, -1.35) arc (1:179:3.75);

\draw(0)--(1)--(2)--(3)--(7)--(6)--(5)--(4)--(1);\draw(7)--(8);\draw(8)--(9)--(11)--(10)--(8);
\draw(a1)--(a2)--(a3)--(a7)--(a6)--(a5)--(a4)--(a1);\draw(a7)--(a8);\draw(a9)--(a8)--(a10);
\draw(b1)--(b2)--(b3)--(b7)--(b6)--(b5)--(b4)--(b1);\draw(b7)--(b8);\draw(b9)--(b8)--(b10);
\draw(0)--(a1);\draw(0)--(b1);\draw(a9)--(11)--(a10);\draw(b9)--(11)--(b10);
\draw(9)--(10);\draw(a9)--(a10);\draw(b9)--(b10);
\draw[thick, dotted] (1,-2.2)--(1,-3.5);\draw[thick, dotted] (6.5,-2.2)--(6.5,-3.2);

\end{tikzpicture}
\caption{\small A graph $G$ such that $\dim(G) - \dim(G-e)$ can be arbitrarily large, where $k \ge 2$.}\label{fig_dim_edge}
\end{figure}
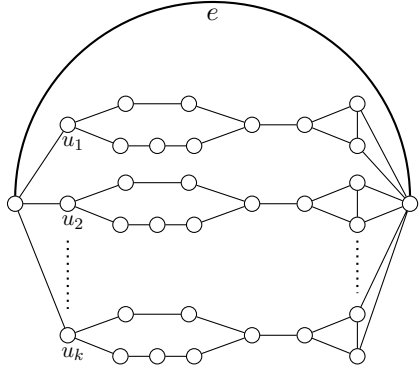

Now, we consider the effect of edge deletion on adjacency dimension. We begin with the following lemma, which is used in proving Theorem~\ref{adim_edge_deletion}. 

\begin{lemma}\label{adim_edge}
For any graph $G$, let $e=xy\in E(G)$.
\begin{itemize}
\item[(a)] If $S$ is an adjacency resolving set of $G$, then $S \cup\{x,y\}$ is an adjacency resolving set of $G-e$.
\item[(b)] If $R$ is an adjacency resolving set of $G-e$, then $R \cup\{x,y\}$ is an adjacency resolving set of $G$.
\end{itemize}
\end{lemma}

\begin{proof}
Let $e=xy \in E(G)$.

(a) Since $S$ is an adjacency resolving set of $G$, $S'=S\cup\{x,y\}$ is also an adjacency resolving set of $G$. Since the adjacency code of each vertex, excluding $x$ and $y$, with respect to $S'$ in $G$ remains the same in $G-e$, $S'$ is an adjacency resolving set of $G-e$.

(b) Since $R$ is an adjacency resolving set of $G-e$, $R'=R\cup\{x,y\}$ is an adjacency resolving set of $G-e$. Since the adjacency code of each vertex, excluding $x$ and $y$, with respect to $R'$ in $G-e$ remains the same in $G$, $R'$ is an adjacency resolving set of $G$.~\hfill
\end{proof}

\begin{theorem}\label{adim_edge_deletion}
For any graph $G$ and any edge $e \in E(G)$,
$\adim(G)-1 \le \adim(G-e) \le \adim(G)+1.$
\end{theorem}

\begin{proof}
We denote by $d_{H,1}(x,y)$ the adjacency distance between two vertices $x$ and $y$ in a graph $H$.

First, we show that $\adim(G-e) \le \adim(G)+1$. Let $S$ be a minimum adjacency resolving set of $G$, and let $e\in E(G)$. Let $x,y \in V(G-e)-S=V(G)-S$ such that $z\in S$ with $d_{G,1}(x,z) \neq d_{G,1}(y,z)$. Without loss of generality, let $d_{G,1}(x,z)=1$ and $d_{G,1}(y,z)=2$; then $xz \in E(G)$. If $d_{G-e,1}(x,z)=d_{G-e,1}(y,z)$, then $e=xz$. Since $z\in S$, $S \cup\{x\}$ forms an adjacency resolving set of $G-e$ by Lemma~\ref{adim_edge}(a). Thus $\adim(G-e) \le |S|+1=\adim(G)+1$.

Second, we show that $\adim(G)-1 \le \adim(G-e)$. Let $R$ be any minimum adjacency resolving set of $G-e$, and let $e=uv\in E(G)$. If $|R \cap \{u,v\}|=0$, then each entry of $a_R(u)$ and $a_R(v)$ is 1 or 2; thus, the adjacency code of each vertex with respect to $R$ in $G-e$ remains the same in $G$, and hence $R$ is an adjacency resolving set of $G$. If $|R \cap\{u,v\}|=1$, say $u \in R$ and $v\not\in R$ without loss of generality, then $R \cup \{v\}$ forms an adjacency resolving set of $G$ by Lemma~\ref{adim_edge}(b). If $|R \cap \{u,v\}|=2$ (i.e., $u,v\in R$), then $R$ is an adjacency resolving set of $G$ by Lemma~\ref{adim_edge}(b). Therefore, $\adim(G) \le |R|+1=\adim(G-e)+1$.~\hfill 
\end{proof}

\begin{remark}
The bounds in Theorem~\ref{adim_edge_deletion} are sharp. 

(a) For a graph $G$ satisfying $\adim(G)-1=\adim(G-e)$, let $G=K_n$ for $n \ge 3$. Then $\adim(G-e)=n-2$ and $\adim(G)=n-1$.

(b) For a graph $G$ satisfying $\adim(G-e)=\adim(G)+1$, let $G$ be the graph in~Figure~\ref{fig_adim_edge}. Let $N(u_1)-\{u_2\}=\cup_{i=1}^{a}\{x_i\}$, $N(u_2)-\{u_1, u_3\}=\cup_{i=1}^{b}\{y_i\}$, and $N(u_3)-\{u_2\}=\cup_{i=1}^{c}\{z_i\}$, where $a,c\ge 3$ and $b \ge 2$.

First, we show that $\adim(G-e)=a+b+c-1$. Let $S$ be a minimum adjacency resolving set of $G-e$. Since any two vertices in $\cup_{i=1}^{a}\{x_i\}$, $\cup_{i=1}^{b}\{y_i\}$, and $\cup_{i=1}^{c}\{z_i\}$, respectively, are twin vertices in $G-e$, by Observation~\ref{obs_twin}(b), we have $|S \cap (\cup_{i=1}^{a}\{x_i\})| \ge a-1$, $|S \cap (\cup_{i=1}^{b}\{y_i\})| \ge b-1$ and $|S \cap (\cup_{i=1}^{c}\{z_i\})| \ge c-1$. Let $S'=(\cup_{i=2}^{a}\{x_i\}) \cup (\cup_{i=2}^{b}\{y_i\})\cup (\cup_{i=2}^{c}\{z_i\}) \subseteq S$. Note that (i) $a_{S'}(u_1)$ is the $(a+b+c-3)$-vector with 1 on the first $(a-1)$ entries and 2 on the rest of the entries; (ii) $a_{S'}(u_2)$ is the $(a+b+c-3)$-vector with 1 on the $a$th through $(a+b-2)$th entries and 2 on the rest of the entries; (iii)  $a_{S'}(u_3)$ is the $(a+b+c-3)$-vector with 2 on the first $(a+b-2)$ entries and 1 on the rest of the entries; (iv) $a_{S'}(x_1)=a_{S'}(y_1)=a_{S'}(z_1)={\bf{2}}_{a+b+c-3}$. Since $S'$ fails to be an adjacency resolving set of $G-e$ and, for any $w \in V(G-e)-S'$, $S' \cup \{w\}$ fails to be an adjacency resolving set of $G-e$, $\adim(G-e)\ge a+b+c-1$. On the other hand, $S' \cup \{x_1, y_1\}$ forms an adjacency resolving set of $G-e$, and hence $\adim(G-e) \le a+b+c-1$. Thus, $\adim(G-e)=a+b+c-1$.

Next, let $e=x_1z_1$. We show that $\adim(G)=a+b+c-2$. By Theorem~\ref{adim_edge_deletion}, $\adim(G)\ge a+b+c-2$. Since $R=(\cup_{i=1}^{a}\{x_i\}) \cup (\cup_{i=2}^{b}\{y_i\})\cup (\cup_{i=2}^{c}\{z_i\})$ forms an adjacency resolving set of $G$ with $|R|=a+b+c-2$, $\adim(G) \le a+b+c-2$. Thus, $\adim(G)=a+b+c-2$. 
\end{remark}

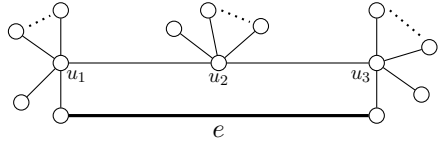
\begin{figure}[ht]
\centering
\begin{tikzpicture}[scale=.7, transform shape]

\node [draw, shape=circle, scale=.8] (0) at  (0, 0) {};
\node [draw, shape=circle, scale=.8] (1) at  (3, 0) {};
\node [draw, shape=circle, scale=.8] (2) at  (6, 0) {};

\node [draw, shape=circle, scale=.8] (01) at  (0, 1) {};
\node [draw, shape=circle, scale=.8] (02) at  (-0.85, 0.6) {};
%\node [draw, shape=circle, scale=.8] (03) at  (-1.1, -0.1) {};
\node [draw, shape=circle, scale=.8] (04) at  (-0.75, -0.75) {};
\node [draw, shape=circle, scale=.8] (05) at  (0, -1) {};

\node [draw, shape=circle, scale=.8] (11) at  (2.15, 0.65) {};
\node [draw, shape=circle, scale=.8] (12) at  (2.8, 1) {};
\node [draw, shape=circle, scale=.8] (13) at  (3.8, 0.7) {};

\node [draw, shape=circle, scale=.8] (21) at  (6, 1) {};
\node [draw, shape=circle, scale=.8] (22) at  (7, 0.3) {};
\node [draw, shape=circle, scale=.8] (23) at  (6.85, -0.6) {};
\node [draw, shape=circle, scale=.8] (24) at  (6, -1) {};

\node [scale=1] at (0.3,-0.25) {$u_1$};
\node [scale=1] at (3,-0.3) {$u_2$};
\node [scale=1] at (5.7,-0.25) {$u_3$};

\node [scale=1.3] at (3,-1.3) {$e$};

\draw(05)--(0)--(1)--(2)--(23);\draw(11)--(1)--(12);\draw(21)--(2)--(22);\draw(01)--(0);\draw[very thick](05)--(24);\draw(0)--(04);\draw(1)--(13);\draw(2)--(24);\draw(02)--(0);%--(03);
\draw[thick, dotted] (-0.6,0.7)--(-0.2,0.9);\draw[thick, dotted] (3.05,1)--(3.59,0.78);\draw[thick, dotted] (6.25,1)--(6.85,0.5);

\end{tikzpicture}
\caption{\small A graph $G$ with $\adim(G-e)=\adim(G)+1$.}\label{fig_adim_edge}
\end{figure}

\begin{question}
Is $\bdim(G-e) \le \bdim(G)+d_{G-e}(u,v)-1$ for any graph $G$, where $e=uv\in E(G)$?
\end{question}

\begin{question}
Is there a family of graphs $G$ such that $\bdim(G)-\bdim(G-e)$ grows arbitrarily large?
\end{question}

\section{Open Problems}\label{s:open}

Below are some open problems about broadcast dimension that are only partially answered by the results in this paper.

\begin{question}\label{q1} What graphs $G$ satisfy $\dim(G)=\bdim(G)$? \end{question}

\begin{question}\label{q2} What graphs $G$ satisfy $\bdim(G)=\adim(G)$? \end{question}

Proposition~\ref{tree_bdim}, Corollary~\ref{tree_b_cor}, and the results in Section \ref{s:highlow} make some progress toward answering Questions \ref{q1} and \ref{q2}.

\begin{question} \label{q3} Is there a family of graphs $G_k$ with $\bdim(G) = k$ for which $\adim(G) = 2^{\Omega(k)}$? \end{question}

Theorem \ref{comparison_3dim} shows that for each $d \geq 1$ there is a family of graphs $G_k$ with $\bdim(G) = k$ for which $\adim(G) =\Omega(k^d)$.

\begin{question}\label{q4} What are the values of $\bdim(T)$ and $\adim(T)$ for every tree $T$?\end{question}

Proposition \ref{tree_bdim} and Corollary \ref{tree_b_cor} make progress on Question \ref{q4}.

\begin{question}\label{q5} Is there a polynomial-time algorithm to determine the value of $\bdim(G)$ for every graph $G$? \end{question}

Heggernes and Lokshtanov \cite{broadcast_complexity2} found a polynomial-time algorithm for computing the broadcast domination number $\gamma_b(G)$. This differs from the standard domination number \cite{NP} and some of its variants \cite{henning, slater0}, which are NP-hard.

Another natural algorithmic problem is to list all minimum resolving broadcasts of a given graph. In the worst-case, any algorithm to solve this problem must take $2^{\Omega(n)}$ time for a graph of order $n$. We find an algorithm that takes $2^{O(n)}$ time to list all minimum resolving broadcasts of any given graph of order $n$.

\begin{theorem}
There is an algorithm that takes $2^{O(n)}$ time to list all minimum resolving broadcasts of any given graph of order $n$. Any algorithm for listing all minimum resolving broadcasts of a given graph of order $n$ must take $2^{\Omega(n)}$ time in the worst-case.
\end{theorem}

\begin{proof}
For the worst-case, note that the graph $H_k$ on $2k$ vertices consisting of $k$ copies of $K_2$ has $2^{\Omega(k)}$ minimum resolving broadcasts, and so does the graph $H'_k$ on $2k+1$ vertices consisting of $k$ copies of $K_2$ and an isolated vertex, so any algorithm for listing all minimum resolving broadcasts of a given graph of order $n$ must take $2^{\Omega(n)}$ on the families $H_k$ and $H'_k$. 

For an algorithm to list all minimum resolving broadcasts of any given graph $G$ of order $n$, we let $v_1, \dots, v_n$ be the vertices of $G$. Let $s = 0$ and perform the following steps:
\begin{enumerate} 
\item Increment $s$. Let $S = \emptyset$. 
\item For each nonnegative integer solution $(x_1, \dots, x_n)$ to the equation $x_1+\dots+x_n = s$, determine if the function $f$ defined by setting $f(v_i) = x_i$ is a resolving broadcast for $G$. If $f$ is a resolving broadcast, add $f$ to $S$. If $S$ is nonempty after checking every solution $(x_1, \dots, x_n)$, return $S$ and halt. Otherwise go back to step $1$.
\end{enumerate}

There are $\binom{s+n-1}{n-1}$ nonnegative integer solutions $(x_1, \dots, x_n)$ to the equation $x_1+\dots+x_n = s$, so the algorithm only checks at most $\binom{2n-2}{n-1} = 2^{O(n)}$ solutions for each value of $s$. For each solution $(x_1, \dots, x_n)$, it takes polynomial time in $n$ to determine whether the solution corresponds to a resolving broadcast for $G$. Thus the algorithm has $2^{O(n)}$ running time for graphs $G$ of order $n$.
\end{proof}

\section*{Acknowledgements}The authors would like to thank Cong X. Kang for his initial discussion on the definition of broadcast dimension. The authors also wish to thank the organizers of the $22^{nd}$ conference of the International Linear Algebra Society, July 8-12, 2019 in Rio de Janeiro, Brazil, where we initially discussed the definition of broadcast dimension.

%%%%%%%%%%%%%%%%%%%%%%%%%%%%%%%%%%%%%%%%%%%%%%%%%
%%%%%%%%%%%%%%%%%%%%%%%%%%%%%%%%%%%%%%%%%%%%%%%%%

\end{document}